\documentclass[12pt]{article}
\usepackage{amsmath,amsfonts,amsthm,amssymb}
\usepackage{newpxtext}
\usepackage{xcolor}
\usepackage{enumitem}
\usepackage{graphicx}
\usepackage{authblk}
\usepackage{latexsym}
\usepackage[mathscr]{eucal}

\newtheorem{theorem}{Theorem}[section]
\theoremstyle{plain}
\newtheorem{acknowledgement}{Acknowledgement}
\newtheorem{corollary}[theorem]{Corollary}
\newtheorem{lemma}[theorem]{Lemma}
\newtheorem{proposition}[theorem]{Proposition}
\numberwithin{equation}{section}
\theoremstyle{definition}

\newtheorem{remark}[theorem]{Remark}
\newtheorem{definition}[theorem]{Definition}
\newtheorem{definitions}[theorem]{Definitions}

\newtheorem{example}[theorem]{Example}

\providecommand{\keywords}[1]
{
  \small	
  \textbf{\textit{Keywords---}} #1
}
\providecommand{\subjclass}[1]
{
  \small	
  \textbf{\textit{AMS subject classification ---}} #1
}

\title{On amalgamation of inverse semigroups over ample semigroups}

\author{Nasir Sohail}

\affil{Institute of Mathematics and Statistics, the University of Tartu, Tartu, Estonia}

\begin{document}
\maketitle

\begin{abstract}
A semigroup amalgam $(S; T_{1}, T_{2})$ is known to be non-embeddable if $T_{1}$ and $T_{2}$ are both groups (completely regular semigroups, Clifford semigroups) but $S$ is not such. We prove some non-embeddability conditions for semigroup amalgams $(S; T_{1}, T_{2})$ in which $T_1$ and $T_2$ are inverse semigroups but $S$ is not such. We also introduce some sufficient conditions that make the inverse hulls of the copies of $S$ in $T_1$ and $T_2$ isomorphic; making, in turn, the amalgam $(S; T_{1}, T_{2})$ weakly embeddable. 
\end{abstract}



\section{Motivation}
The amalgamation problem of semigroups has its origins in the early work of J. M. Howie from the 1960s. (The inpiration came from group amalgams, which were considered earlier by O. Schreier.) The topic was then extensively studied by various mathematicians during the second half of the previous century. References to this work may be found in Howie's celebrated monograph \cite{Howie}, of which the last chapter is also dedicated to semigroup amalgams. The main emphasis, during all these years, had been on determining the embeddability conditions for semigroup amalgams. Non-embeddable amalgams were discovered sporadically, usually as by-products. One of Howie's pioneering articles \cite{Howie 1}, however, provided an important class of non-embeddable amalgams, that may essentially be viewed as groups intersecting in semigroups. Generalizing Howie's result, Rahkema and Sohail \cite{Kristiina and Nasir} came up in 2014 with two more classes of non-embeddable semigroup amalgams. The current article furthers the same line of research of investigating the (non-embeddability of) amalgams that may essentially be viewed as inverse semigroups intersecting in a non-inverse semigroup.\vspace{2pt}

The study of ample semigroups, and their variants, has been an active area of research for many decades, see for instance \cite{Fountain} and its references. As every ample semigroup $S$ gives rise to an amalgam $(S;T_{1},T_{2})$, where $T_{1}$ and $T_{2}$ are inverse semigroups, it was natural for us to consider the amalgams $(S;T_{1},T_{2})$, such that $T_{1}$ and $T_{2}$ are inverse semigroups and $S$ belongs to certain class of ample semigroups. In fact, we introduce, in this connection, the classes of rich and ultra-rich ample semigroups; the intersection of the latter class with the class of inverse semigroups is precisely the class of all groups. 

\section{Introduction and preliminaries}

Given a semigroup $S$, an element $x\in S$ is called \textit{invertible} if there
exists a unique element $x^{-1}\in S$ such that $xx^{-1}x=x$ and $
x^{-1}xx^{-1}=x^{-1}$. We call $S$ an \textit{inverse semigroup} if every $x\in S$ is invertible. \textit{Inverse monoids} are defined similarly. Let $X$ be a non-empty set. Then the set $\mathcal{I}_{X}$ of all partial bijections of $X$ is an inverse semigroup under the usual composition of partial maps. We call $\mathcal{I}_{X}$ the \textit{symmetric inverse semigroup} over $X$. By the Wagner-Preston representation theorem (see, for instance, \cite{Howie} Theorem 5.1.7) any inverse semigroup $S$ can be embedded in the symmetric inverse semigroup $\mathcal{I}_{S}$. If $S$ is a subsemigroup of an inverse semigroup $T$ then the inverse subsemigroup of $T$ generated by $S$ is called the \textit{inverse hull} of $S$ in $T$. \textit{Homomorphisms} of inverse semigroups (monoids) are precisely the semigroup homomorphisms. Throughout this article, the maps will be written to the right of their arguments. Also, we shall omit parentheses around the arguments if there is no risk of confusion. For further details about inverse semigroups, as well as standard definitions in semigroup theory, the reader may refer to the texts \cite{Howie, Lawson}.\vspace{2pt}

A semigroup $S$ is called \textit{right ample} if it can be embedded in an inverse semigroup $T$ (typically, in the symmetric inverse semigroup $\mathcal{I}_{X}$ of a non-empty set $X$) such that the image of $S$ is closed under the unary operation $s \longmapsto s ^{-1}s $, where $S$ is identified with its isomorphic copy in $T$ and $s^{-1}\in T$ denotes the inverse of $s \in S$. We shall call $T$ an inverse semigroup \textit{associated} with $S$. \textit{Left ample} semigroups are defined analogously. We say that $S$ is \textit{ample} if it is both right and left ample.  If $S$ is a subsemigroup of an associated inverse semigroup $T$ then we shall say that $S$ is (\textit{right}, \textit{left}) \textit{ample in} $T$. Every full subsemigroup of an inverse semigroup $S$ is ample in $S$. The converse is not true; for example, $\mathbb{N}$ is ample but not full in the multiplicative monoid $\mathbb{Q}$. It is possible that $S$ is made into a left and a right ample semigroup by different associated inverse semigroups. In such a case, the problem of finding a single (associated) inverse semigroup making $S$ into a left as well as right ample semigroup is, in general, undecidable (\cite{Gould and Kambites}, Theorem 3.4 and Corollary 4.3). If $T_{1}$ and $T_{2}$ are inverse semigroups admitting a homomorphism $\phi:T_{1}\longrightarrow T_{2}$ and $S$ is right (respectively, left) ample in $T_{1}$, then one can easily verify that $S\phi$ is right (respectively, left) ample in $T_{2}$. For further details about ample semigroups, the reader is referred to \cite{Fountain} and the references contained therein.\vspace{2pt}

A semigroup \textit{amalgam} is a $5$-tuple $\mathcal{A\equiv }$ $(S;T_{1},T_{2};\phi _{1},\phi _{2})$ comprising pair-wise disjoint semigroups $S$, $T_{1}$, $T_{2}$ and monomorphisms
\begin{equation*}
\phi _{i}:S\longrightarrow T_{i}\text{, }1 \leq i\leq 2 \text{.}
\end{equation*}
We say that $\mathcal{A}$ is \textit{embeddable} (or \textit{strongly embeddable}, for emphasis) if there exists a semigroup $T$ admitting monomorphisms $\psi _{i}:T_{i}\longrightarrow T$, $1 \leq i\leq 2$, such that 

\begin{enumerate}
\item[$(i)$] $\phi _{1}\psi _{1}=\phi _{2}\psi _{2}$,
\item[$(ii)$] $\forall ${\tiny \ }$t_{1}\in T_{1}$, $\forall ${\tiny \ }$
t_{2}\in T_{2}$, $t_{1}\psi _{1}=t_{2}\psi _{2}$ $\Longrightarrow $ $\exists 
$ $s\in S$ such that $t_{1}=s\phi _{1}$, $t_{2}=s\phi _{2}$.
\end{enumerate}

\noindent If Condition $(ii)$ is not necessarily satisfied, then $\mathcal{A}$ is said to be \textit{weakly embeddable}. We call $(S;T_{1},T_{2};\phi _{1},\phi _{2})$ a \textit{special amalgam} if $T_{1}$ and $T_{2}$ are isomorphic, say, via $\nu :T_{1}\longrightarrow T_{2}$, such that $ s \phi _{1}\nu
=s \phi _{2}$ for all $s\in S$. Any special amalgam is weakly embeddable; for instance, in $T_{1}$. It is customary to denote a semigroup amalgam by $(S;T_{1},T_{2})$ if no explicit mention of $\phi_{1}$ and $\phi _{2}$ is needed. We shall also call $(S;T_{1},T_{2})$ an \textit{amalgam over} $S$. Every ample semigroup $S$ gives rise to an amalgam $(S;T_{1},T_{2})$ in which $S$ is right (respectively, left) ample in the inverse semigroup $T_{i}$ (respectively, $T_{j}$), where $\{i,j\}=\{1,2\}$. We shall consider these amalgams in Theorem \ref{Non-embeddable2}. \vspace{2pt}

Let $T_{1}\ast T_{2}$ denote the free product of semigroups $T_{1}$ and $T_{2}$. Then by the \textit{amalgamated coproduct} of $(S;T_{1},T_{2};\phi_{1},\phi _{2})$ we mean the quotient semigroup $\left( T_{1}\ast T_{2}\right) /\theta _{R}$, where $\theta _{R}$ denotes the congruence on $T_{1}\ast T_{2}$ generated by the relation
\begin{equation*}
   R =\{\left(s\phi _{1},s\phi _{2}\right) :s\in S\}.  
\end{equation*}
We denote $\left( T_{1}\ast
T_{2}\right) /\theta _{R}$ by $T_{1}\ast _{S}T_{2}$. In fact, the following diagram is a pushout in the category of all semigroups, where the homomorphisms
\begin{equation*}
    \eta _{i}: T_{i}\longrightarrow T_{1}\ast _{S}T_{2}, \,1\leq i \leq 2,
\end{equation*}
 send $x\in T_{i}$ to the congruence class $(x)_{\theta_R}\in T_{1}\ast _{S}T_{2}$.
 
\begin{figure}[htbp]
\centering
\includegraphics [width=50mm]{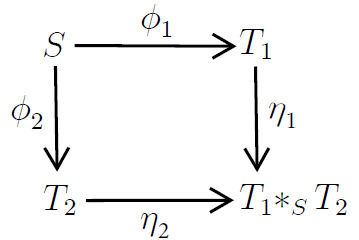}
\caption{Amalgamated coproduct}\label{D1}
\end{figure}

 \begin{theorem}[\cite{Howie}, Theorem 8.2.4]
\label{product}A semigroup amalgam $(S;T_{1},T_{2})$ is (weakly) embeddable if and only if it is (weakly) embedded in $T_{1}\ast _{S}T_{2}$ via the homomorphisms $\eta_{i}:T_{i}\longrightarrow T_{1}\ast _{S}T_{2}$, $i\in \{1,2\}$, defined above.
\end{theorem}

\begin{proof}
    Follows immediately from the properties of a pushout.
\end{proof}

A semigroup $S$ is called an \textit{amalgamation base} for a class (equivalently, category) $\mathcal{C}$ of semigroups if every amalgam $(S;T_{1},T_{2})$, with $T_{1},T_{2}\in \mathcal{C}$, is embeddable in some $T\in \mathcal{C}$. Given a semigroup $T_{1} \in \mathcal{C}$ containing an isomorphic copy of a semigroup $S$, we say that $(S;T_{1})$ is an \textit{
amalgamation pair} for $\mathcal{C}$ if for all $T_{2}\in \mathcal{C}$ the amalgam $(S;T_{1},T_{2})$ is embeddable in some $T\in \mathcal{C}$. \textit{Weak amalgamation bases} (\textit{pairs}) are defined similarly.

\begin{theorem}[\protect\cite{Howie}, Theorems 8.6.1 and 8.6.4]
\label{Inv1}Inverse semigroups are amalgamation bases for the classes of all semigroups and inverse semigroups.
\end{theorem}

\noindent Let $S$ be an inverse and $T$ be an arbitrary semigroup. Then, by the above theorem, $(S;T)$ is an amalgamation pair for
the class of all semigroups. If $S$ and $T$ are both inverse then $(S;T)$ is also an amalgamation pair for the class of inverse semigroups. All of the assertions made in this section about semigroups are also true for monoids.

\section{Amalgamation over ample semigroups}

It was shown by Howie \cite{Howie 1} that a semigroup amalgam $(S;T_{1},T_{2})$ does not embed if $T_{1}$ and $T_{2}$ are both groups but $S$ is not such. Generalizing this result, Rahkema and Sohail \cite{Kristiina
and Nasir} showed that $(S;T_{1},T_{2})$ is non-embeddable if $T_{1}$ and $T_{2}$ are both completely regular (respectively, Clifford) semigroups but $S$ is not completely regular (respectively, Clifford). In this section we shall consider the amalgams 
$(S;T_{1},T_{2})$ in which $T_{1}$ and $T_{2}$ are both inverse semigroups but $S$
is not such---the non-embeddability of such amalgams was left as an open problem in \cite{Kristiina
and Nasir}.

\begin{definition}
Let $\mathcal{C}$ be a class of semigroups. Suppose that $T_{1}\in \mathcal{C}$ contains an isomorphic copy of a semigroup $S$ via $\phi_{1}:S \longrightarrow T_{1}$. Then the pair $(S;T_{1})$ will be called an \textit{antiamalgamation pair} for $\mathcal{C}$ if for every $T_{2}\in \mathcal{C}$ and every monomorphism $\phi_{2}:S\longrightarrow T_{2}$ the amalgam $(S;T_{1},T_{2};\phi_{1}, \phi_{2})$ is non-embeddable (in any semigroup).
\end{definition}

\begin{theorem}
\label{Non-embeddable1}Let $T_{1}$ be an inverse semigroup and $\phi_{1}:S \longrightarrow T_{1}$ be a monomorphism such that $S\phi_{1}$ is right as well as left ample in $T_{1}$. If $S$ is non-inverse then $(S; T_{1})$ is an antiamalgamation pair for the class of inverse semigroups.
\end{theorem}

\begin{proof}
Let $S$, $T_{1}$ and $\phi_{1}$ be as described in the statement of the theorem. Let $T_{2}$ be an inverse semigroup admitting a monomorphism $\phi _{2}:S \longrightarrow T_{2}$. Given $s \in S$, let us denote $s\phi_1$ and $s \phi_2$ by $s_1$ and $s_2$, respectively. Identifying $S$ with its isomorphic copies $S \phi_1$ and $S\phi_2$, and using the properties of inverses, we may calculate in $T_1 *_S T_2$:
\begin{equation}\label{IdmptntsinPO}
\begin{split}
ss_1^{-1}ss_2^{-1}&=ss_2^{-1}, \hspace{10pt} ss_2^{-1}ss_1^{-1}=ss_1^{-1},\\
s_2^{-1}ss_1^{-1}s&=s_2^{-1}s, \hspace{10pt} s_1^{-1}ss_2^{-1}s=s_1^{-1}s.
\end{split}
\end{equation}
Since $S\phi_1 $ is right and left ample in $T_1$, the identification of $S$ with $S\phi_1$ also gives,

\begin{equation*}
    ss_1^{-1},\,s_1^{-1}s \in S.
\end{equation*}
By the commutativity of idempotents in $T_2$, we may write from (\ref{IdmptntsinPO}):

\begin{equation}\label{IdmptntsinPO1}
ss_1^{-1}=ss_2^{-1}, \hspace{10pt} s_1^{-1}s=s_2^{-1}s.
\end{equation}
Now, using (\ref{IdmptntsinPO1}), we calculate in $T_1*_ST_2$:
\begin{equation}\label{IdmptntsinPO2}
s_1^{-1}=s_1^{-1}(ss_1^{-1})=s_1^{-1}(ss_2^{-1})= (s_1^{-1}s)s_2^{-1}=(s_2^{-1}s)s_2^{-1} = s_2^{-1}.
\end{equation}
Because $S\phi_1$ and $S\phi_2$ are non-inverse, there exists $s \in S$ such that $s_i^{-1} \notin T_i,\, 1\leq i \leq 2$. The amalgam $(S;T_1,T_2)$, therefore, fails to embed by (\ref{IdmptntsinPO2}).
\end{proof}

\begin{example}\label{Example3.3}
\upshape Let $\mathbb{L}$ denote the lattice of ample submonoids of the symmetric inverse semigroup $\mathcal{I}_{n}$ over a finite chain $C_{n}: 1<2<\cdots <n$, given in \cite{NSA}, Figure \ref{D1}. The chain $\mathbb{L}_{\text{INV}}:\{\iota\} \subseteq \mathcal{OI}_{n}\subset \mathcal{RI}_{n}\subset \mathcal{I}_{n}^{\prime }\subset \mathcal{I}_{n}$ constitutes the sublattice of $\mathbb{L}$ comprising the inverse submonoids of $\mathcal{I}_{n}$. This gives a (finite) set
\begin{equation*}
\{(S,T):S\in \mathbb{L\smallsetminus L}_{\text{INV}}~,T\in \mathbb{L}_{\text{INV}} \text{ with } S \subseteq T \mathcal{\}}
\end{equation*}
of antiamalgamation pairs for the class of inverse semigroups.
\end{example}

\begin{theorem}
\label{Non-embeddable2}Let a non-inverse semigroup $S$ be made into a right (respectively, left) ample semigroup by an associated inverse semigroup $T_{1}$ (respectively, $T_{2}$). Then the amalgam $(S;T_{1},T_{2})$ is not embeddable (in any semigroup).
\end{theorem}

\begin{proof}
Let $S$, $T_1$ and $T_2$ be as given in the statement of the theorem. Then, as before, the identification of $S$ with its isomorphic copies in $T_1$ and $T_2$ gives equations (\ref{IdmptntsinPO}). Since $S$ is right ample in $T_1$ and left ample in $T_2$, we have $s_1^{-1}s,\, ss_2^{-1} \in S$. Subsequently, $ss_1^{-1}$, $ss_2^{-1}$ commute in $T_1$ and $s_1^{-1}s$, $s_2^{-1}s$ commute in $T_2$. Using the argument from the proof of Theorem \ref{Non-embeddable1}, we can once more deduce equations  (\ref{IdmptntsinPO1}) from equations  (\ref{IdmptntsinPO}). However, equations (\ref{IdmptntsinPO1}) give $s_1^{-1}=s_2^{-1}$ implying (as in the said proof) that $(S;T_1,T_2)$ is non-embeddable.
\end{proof}

\section{Weak amalgamation}

Given a subsemigroup $S$ of an inverse semigroup $T$, we define its \textit{dual} to be the subsemigroup
\begin{equation*}
S^{\prime }=\{s^{-1}\in T:s\in S\}. 
\end{equation*}
Defining $\alpha :S \longrightarrow S'$ by $ s\mapsto s^{-1}$, we have

\begin{equation*}
(xy)\alpha=\left( xy\right) ^{-1}=y^{-1}x^{-1}= (y)\alpha
(x)\alpha,\forall\, x,y\in S,
\end{equation*}

\noindent implying that $S$ and $S'$ are \textit{anti-isomorphic}. Clearly, if non-empty, $S\cap S^{\prime }$ is an inverse subsemigroup of $S$ and $S^{\prime
}$ with $E(S)=E(S')\subseteq S\cap S^{\prime }$, where $E(U)$ denotes the set of idempotents of any semigroup $U$. Also, if $S$ is right (respectively, left) ample in $T$ then $S'$ is a left (respectively, right) ample subsemigroup of $T$.

\begin{lemma} \label{Rmrk3.1}
    Let $T_{1}$ and $T_{2}$ be inverse semigroups containing isomorphic copies, say $S_{1}$ and $S_{2}$, of a semigroup $S$. Then, there exists a bijection between the sets $S_{1} \cup S'_{1}$ and $S_{2} \cup S'_{2}$. Moreover, for any $x\in S_{1} \cup S'_1$ one has:
    \begin{equation*}
        (x^{-1})\psi=(x\psi)^{-1}.
    \end{equation*}
\end{lemma}

\begin{proof}
    Let $\phi$ be the isomorphism from $S_{1}$ to $S_{2}$. Then $\phi'= \alpha_{1}^{-1} \circ \phi \circ \alpha_{2}$ is
    an isomorphism from $S'_{1}$ to $S'_{2}$, where $\alpha_{i}:S_{i} \longrightarrow S'_{i}$, $i=1,2$, are the anti-isomorphisms defined by $s_{i} \mapsto s_{i}^{-1},\, s_{i} \in S_{i}$. Let $x \in S_{1} \cap S'_{1}$. Then $x^{-1}\in S_{1} \cap S'_{1}$, and, in particular, $x,x^{-1} \in S_{1}$. Using the assumption that $\phi$ is an isomorphism, we have:
    \begin{equation*}
    \begin{split}
x\phi&=(xx^{-1}x)\phi=(x)\phi(x^{-1})\phi(x)\phi, \vspace{6pt}\\
(x^{-1})\phi&=(x^{-1}xx^{-1})\phi =(x^{-1})\phi(x)\phi(x^{-1})\phi.
    \end{split}
\end{equation*}
    
\noindent Hence, we have $(x^{-1})\phi=((x)\phi)^{-1}$ whenever $x \in S_{1} \cap S'_{1}$. We can now calculate:
    
    \begin{equation*}    x\phi=((x^{-1})^{-1})\phi= ((x^{-1})\phi)^{-1}= x \phi'.
    \end{equation*}
    This implies that $\phi$ and $\phi'$ agree on $S_{1} \cap S'_{1}$. Consequently, the map
    \begin{equation*}
        \psi=\phi \cup \phi':S_{1} \cup S'_{1} \longrightarrow S_{2} \cup S'_{2}
    \end{equation*}
    is well-defined. Using a similar argument one may also construct
    \[\psi^{-1}:S_2 \cup S'_2 \longrightarrow S_2 \cup S'_2,
    \]
    such that $\psi \circ \psi'^{-1}$ and $\psi'^{-1}\circ \psi$ are both identity functions. This implies that $\psi$ is a bijection, as required.\vspace{2pt}

    Concerning the second part of the lemma, we have seen above that $(x^{-1})\phi=((x)\phi)^{-1}$ for all $x \in S_1$. Similarly, it can be shown that $((x')^{-1})\phi'=((x)\phi')^{-1}$ for all $x' \in S'_1$. This gives $(x^{-1})\psi = (x\psi)^{-1}$, and the proof is complete.
\end{proof}

\begin{proposition}\label{VG}
    Let $S$ be any semigroup and $T_{i}$, $1 \leq i \leq 2$, be inverse semigroups admitting monomorphisms 
    $\phi_{i}:S \longrightarrow     T_{i}$. 
    Then the amalgam $(S, T_{1}, T_{2}; \phi_{1}, \phi_{2})$ is weakly embeddable in an inverse semigroup if and only if $(S;V_{1},V_{2})$ constitutes a special amalgam, where $V_{i}$ is the inverse hull of  $S\phi_{i}$ in $T_{i}$.
\end{proposition}
\begin{proof}
    ($\implies$) Let $S$, $T_{1}$, $T_{2}$, $V_{1}$, $V_{2}$, and $\phi_{1}$, $\phi_{2}$ be as described in the statement of the theorem. Assume that $(S;T_{1},T_{2})$ is weakly embeddable in an inverse semigroup $W$, via monomorphisms $\mu_{1}:T_{1} \longrightarrow W$ and $\mu_{2}:T_{2} \longrightarrow W$. We shall denote $S\phi_{i},\, 1 \leq i \leq 2,$ by $S_{i}$.\vspace{2pt}
    
    Observe that any element of $V_{1}$ may be written in the form $x_{1} x_{2} \cdots x_{n}$, where $x_{1}, x_{2}, \dots, x_{n} \in S_{1} \cup S'_{1}$ and for all $1\leq i \leq n-1$ the elements $x_{i}, x_{i+1}$ are not both in $S_{1}$ or $S'_{1} \smallsetminus S_{1}$.  Similarly, the elements of $V_{2}$ can be written as
    $y_{1}y_{2} \cdots y_{m}$,
    where $y_{1},y_{2},\dots, y_{m} \in S_{2} \cup S'_{2}$ and for all $1\leq i \leq m-1$ the elements $y_{i}, y_{i+1}$ do not both belong to $S_{2}$ or $S'_{2} \smallsetminus S_{2}$. Also, for each $i \in \{1,2\}$ and $x \in S_i$ we have:
    \[
    (x^{-1})\mu_i = (x\mu_i)^{-1}, \hspace{6pt} \text{where } \,x^{-1} \in S'_i.
    \]
    We define $\theta:V_{1} \longrightarrow V_{2}$ by
    \begin{equation*}
        (x_{1}x_{2} \cdots x_{n})\theta = (x_{1}x_{2} \cdots x_{n})\mu_1 \mu_2^{-1}.
    \end{equation*}
    Then, $\theta$ is clearly an isomorphism from $V_1$ to $V_2$.\vspace{4pt}
    
\noindent ($\impliedby$) Let $(S; V_{1}, V_{2};\phi_{1},\phi_{2})$ be made into a special amalgam by the isomorphism $\nu:V_{1} \longrightarrow V_{2}$. Then
\begin{equation}\label{3.4}
    \phi_{1} \circ \nu = \phi_{2}.
\end{equation}\vspace{2pt}

\noindent Consider a semigroup $V$ admitting isomorphisms $\gamma_{i}:V \longrightarrow V_{i}$, for each $1 \leq i \leq 2$, with $V \cap V_{i} = \emptyset$, and
\begin{equation}\label{3.5}
\gamma_{1} \circ \nu = \gamma_{2};
\end{equation}
that is $(V;V_{1},V_{2})$ is a special amalgam. Then, being an inverse semigroups amalgam, $(V;T_{1}, T_{2}; \gamma_{1}, \gamma_{2})$ is embeddable in an inverse semigroup, say $W$, via monomorphisms, say, $\mu_{i}:T_{i}\longrightarrow W$. This implies that
\begin{equation}\label{3.6}
    \gamma_{1} \circ \mu_{1} = \gamma_{2} \circ \mu_{2}.
\end{equation}
Now, using (\ref{3.4}) and (\ref{3.5}), we have:
\begin{equation}\label{3.7}
 \phi_{1} \circ \gamma_{1}^{-1}=\phi_{1} \circ (\nu \circ \gamma_{2}^{-1})=(\phi_{1} \circ \nu) \circ \gamma_{2}^{-1}=\phi_{2} \circ \gamma_{2}^{-1}.
\end{equation}
Finally, using (\ref{3.6}) and (\ref{3.7}), we may calculate:
\begin{equation*}
    \phi_{1} \circ \mu _{1} =\phi_{1}\circ \gamma_{1}^{-1} \circ \gamma_{1} \circ \mu_{1} = \phi_{2} \circ \gamma_{2}^{-1} \circ \gamma_{2} \circ \mu_{2} =\phi_{2} \circ \mu_{2}.
\end{equation*}
Hence, $(S;T_{1},T_{2})$ is weakly embeddable.
\end{proof}

\subsection{Weak amalgamation over rich ample semigroups}

In this subsection, we introduce the notion of rich (right, left) ample semigroups. Given inverse semigroups (respectively, groups) $T_{1}$ and $T_{2}$, we show that an amalgam $(S;T_{1},T_{2})$ is weakly embeddable in an inverse semigroup (respectively, group) if $S$ is rich ample in $T_{1}$ and $T_{2}$. We begin by recalling that any inverse semigroup $S$ comes equipped with the \textit{natural partial order}:
\begin{equation*}
    \forall x,y \in S,\, x \leq y\, \text{ iff }\, x=ey, \text{ for some }\, e \in E(S).
\end{equation*}

\begin{remark} \label{min}
    Let $U$ be an inverse semigroup. Then $uu^{-1}$ is the minimum idempotent, with respect to the natural partial order, such that $(uu^{-1})u=u$. To see this, let $eu=u$ for some idempotent $e \in U$. Then $u^{-1}e=u^{-1}$, and we have
\begin{equation*}
    uu^{-1}=euu^{-1}e=uu^{-1}e.
\end{equation*}
This implies that $uu^{-1} \leq e$, and hence the assertion.
\end{remark}

\begin{definitions}\label{RRA&RLA}
A subsemigroup $S$ of an inverse semigroup $T$ is called \textit{rich
right ample} in $T$, if
\begin{equation*}
\forall x,y \in S,\, x^{-1}y \in S \cup S^{\prime },
\end{equation*}
where $S'=\{z^{-1}\in T:z\in S\}$ is the dual of $S$. We say that $S$ is \textit{rich left ample} in $T$ if 
\begin{equation*}
\forall x,y \in S,\, xy^{-1} \in S \cup S^{\prime }.
\end{equation*}
A subsemigroup of $T$ is called \textit{rich ample} in $T$ if it is both rich right and rich left ample in $T$. A submonoid $S$ of an inverse semigroup $T$ is rich (right, left) ample in $T$ if it is such as a subsemigroup. By saying that $S$ is a rich (left, right) ample subsemigroup of $T$ we shall mean that $S$ is rich (right, left) ample in $T$.

\end{definitions}

\begin{lemma}\label{U&U'}
Let $S,\, S'$ and $T$ be as defined above. Then
\begin{enumerate}
    \item $S$ is rich right (respectively, left) ample in $T$ if and only if (its dual) $S'$ is a rich left (respectively, right) ample in $T$,
    \item $S$ is (rich) ample in $T$ if and only if $S'$ is (rich) ample in $T$.
\end{enumerate}
\end{lemma}

\begin{proof}
The proof is straightforward.
\end{proof}

\begin{proposition}\label{Invhull1}
\label{P1}  Let $S$ be a subsemigroup of an inverse semigroup $T$. Then $S$ is rich ample in $T$ if and only if the inverse hull of{\tiny \ \ \ }$S$ in $T$ equals $S \cup S^{\prime }$.
\end{proposition}

\begin{proof}
    The proof is straightforward.
\end{proof}

\begin{lemma}\label{4.6}
Let $S$ be a rich ample subsemigroup of an inverse semigroup $T$. Then $S$ and $S'$ are the down closed subsemigroups of $S \cup S'$ with respect to the natural partial order.
\end{lemma}

\begin{proof}
    Let $s \in S$ and $s' \in S'$ be such that $s' \leq s$, where $\leq$ denotes the natural partial order on the inverse semigroup $S \cup S'$. Then $s'=es$ for some idempotent $e \in E(S \cup S')$. Now, because $e \in S \cap S'$, it follows (in particular) from $e \in S$ that $s'=es \in S$. Similarly, one can show that $s \leq s'$ implies that $s \in S'$.
\end{proof}

\begin{proposition}
\label{Prop1}  A subsemigroup $S$ of an inverse semigroup $T$ is rich right ample if and only if 

\begin{enumerate}
\item $S$ is right ample in $T$, and 

\item for all $x,y\in S$, $x^{-1}y=x^{-1}xa$, \,where $a\in S\cup S^{\prime }$.
\end{enumerate}
\end{proposition}

\begin{proof}
    ($\implies$) Let $S$ be rich right ample in $T$. Then for all $x \in U$, the element $x^{-1}x$ belongs to $S \cup S'$. Because $x^{-1}x$ is an idempotent, we in fact have $x^{-1}x \in S \cap S' \subseteq S$, meaning that $S$ is right ample in $T$. Next, let $x,y \in S$ and observe that
    \[    x^{-1}y=x^{-1}xx^{-1}y=x^{-1}xa,
    \]
    where $a=x^{-1}y \in S \cup S'$. Hence the second condition is also satisfied. \vspace{2pt}
    
    \noindent ($\impliedby$) Let $S$ be a subsemigroup of an inverse semigroup $T$ satisfying Conditions (1) and (2) of the proposition. Let $x,y\in S$. Then by the second condition, $x^{-1}y=x^{-1}xa$, where $a \in S \cup S'$. Now, $x^{-1}x \in S \cap S'$, because, by the first condition, $S$ is right ample. This implies that
    \begin{equation*}
        x^{-1}y=x^{-1}xa \in S \cup S'\text{,}
    \end{equation*}
    whence $S$ is rich right ample in $T$.
\end{proof}

\begin{proposition}
\label{Prop2}  A subsemigroup $S$ of an inverse semigroup $T$ is rich left ample if and only if 

\begin{enumerate}
\item $S$ is left ample in $T$, and 

\item for all $x,y\in S$, $xy^{-1}=byy^{-1}$, \,where $b\in S\cup S^{\prime }$
. 
\end{enumerate}
\end{proposition}

\begin{proof}
Similar to the proof of the above proposition.
\end{proof}

\begin{remark}\label{4.8}
    If $\phi: T_{1} \longrightarrow T_{2}$ is a homomorphism of inverse semigroups and $S$ is rich right (left) ample in $T_{1}$ then it can be easily verified that such is $S\phi$ in $T_{2}$.
\end{remark}

\begin{theorem}\label{1RAmple}
    Let $T_{1}$ and $T_{2}$ be inverse semigroups admitting monomorphisms $\phi_{1}$ and $\phi_{2}$ from a semigroup $S$, respectively, such that $S\phi_{1}$ is rich right ample in $T_{1}$ but $S\phi_{2}$ is not such in $T_{2}$. Then the amalgam $(S;T_{1}, T_{2})$ fails to embed weakly in any inverse semigroup.
\end{theorem}

\begin{proof}
   Let $S, T_1, T_2, \phi_1$ and $\phi_2$ be as described in the statement of the Theorem. Assume, on the contrary, that $(S;T_1,T_2)$ is weakly embeddable in an inverse semigroup. Let $(S;V_1,V_2)$ be the special amalgam considered in Proposition \ref{VG}. Clearly, $S_1 = S\phi_1$ is rich right ample in $V_1$. But then its image $S_1\mu = S\phi_2$ must be such in $V_2$ and hence in $T_2$, a contradiction.
\end{proof}
    
\begin{remark}
    The dual statement, obtained by replacing 'rich right ample' with 'rich left ample' in Theorem \ref{1RAmple}, can be proved on similar lines.
\end{remark}

\begin{lemma}\label{3.14}
    Let $T_{1}$ and $T_{2}$ be inverse semigroups containing isomorphic copies $S_{1}$ and $S_{2}$ of a semigroup $S$, respectively, that are rich ample in the respective oversemigroups. Then the postes $S_{1} \cup S'_{1}$ and $S_{2} \cup S'_{2}$ are order-isomorphic.
\end{lemma}
\begin{proof}
    Let the map
    \begin{equation*}
        \psi=\phi \cup \phi':S_{1} \cup S'_{1} \longrightarrow S_{2} \cup S'_{2} 
    \end{equation*}
    be as defined in Lemma \ref{Rmrk3.1}. Then, it follows from the Lemma \ref{4.6} that $\psi$ is indeed an order-embedding.
\end{proof}

\begin{theorem}\label{4.9}
    Let $T_{1}$ and $T_{2}$ be inverse semigroups containing isomorphic copies, say $S\phi_{1}$ and $S\phi_{2}$, respectively, of a semigroup $S$. Assume also that $S\phi_{i}$ is rich ample in $T_{i}$ for each $ i \in \{1,2 \}$. If $S$ is not inverse then the amalgam $(S;T_{1}, T_{2}; \phi_{1},\phi_{2})$ is weakly (but not strongly) embeddable in an inverse semigroup.
\end{theorem}

\begin{proof}
    Let $S_{1}=S\phi_{1}$ and $S_{2}=S \phi_{2}$. Then, by Proposition \ref{Invhull1}, the inverse hull of $S_{i}$ in $T_{i}$, $1 \leq i \leq 2$, is $S_{i} \cup S'_{i}$.  The main objective will be proving that $S_{1}\cup S'_{1}$ and 
    $S_{2}\cup S'_{2}$ are isomorphic. We shall prove, to this end, that the poset order-isomorphism $\psi:S_{1} \cup S'_{1} \longrightarrow S_{2} \cup S'_{2}$, considered in Lemma \ref{3.14}, is a homomorphism of semigroups.
    
    Clearly, if $x,y \in S_{1}$ (equivalently, $S'_{1}$) then $\psi(xy)=\psi(x)\psi(y)$.
    We prove that $(xz)\psi = (x) \psi (z) \psi$, for all $x \in S_1$ and $z \in S'_1$---that $(zx)\psi = (z)\psi (x)\psi$ will follow from the symmetry. If $x \in S_{1}$ and $z \in S'_{1}$ then $z=y^{-1}$ for some $y \in S_{1}$ and we observe, by Proposition \ref{Prop2}, that
    
    \begin{equation*}    (xz)\psi = (xy^{-1})\psi=(byy^{-1})\psi=(b)\psi(yy^{-1})\psi, \,\, \text{where }\, b\in S \cup S'.
    \end{equation*}
    
    \noindent We first show that $(yy^{-1})\psi= (y)\psi(y^{-1})\psi$. To this end, let us first recall from Remark \ref{min} that
    \begin{equation}\label{3.1}
        yy^{-1}=\text{min}\,\{e \in E_{1} : ey=y\},
    \end{equation}
    where $E_{1}=E(S_{1} \cup S'_{1})$. Then, we note that the restriction of $\psi$ to $E_{1}$ is identity and $ey=y$ for $e \in E_{1}$ if and only if $(ey)\psi=(e\psi) (y\psi) = y\psi$. Thus
    \begin{equation}\label{3.2}
        (\{e \in E_{1}: ey =y\})\psi=\{e\psi \in E_{1}\psi:  (e)\psi(y)\psi = y\psi\}.
    \end{equation}
    Now, recall from Lemma \ref{3.14} the map $\psi:S_{1} \cup S'_{1} \longrightarrow S_{2} \cup S'_{2}$ is an order-isomorphism of posets, whence we may write from (\ref{3.1}) and (\ref{3.2})
    \begin{equation*}
    \begin{split}
        (yy^{-1})\psi&= (\text{min}\{e \in E_{1}: ey =y\})\psi
        \vspace{6pt} \\
        &=\text{min}(\{e \in E_{1}: ey =y\})\psi
        \vspace{6pt} \\
        &=\text{min}\, \{e\psi \in E_{1}\psi:  (e)\psi(y)\psi = y\psi\}
        \vspace{6pt} \\
        &= (y)\psi (y\psi)^{-1} \text{,} \hspace{15pt}\text{since } E_{1}\psi = E(S_{2} \cup S'_{2}) \text{,}\\
        &= (y)\psi(y^{-1})\psi \text{,} \hspace{15pt}\text{by Lemma } \ref{Rmrk3.1}.
    \end{split}
    \end{equation*}
    
Coming back to proving that $(xz)\psi = (x)\psi (z) \psi$, we consider, in view of Definitions \ref{RRA&RLA}, two cases.\vspace{2pt}

\noindent If $ xy^{-1} = b \in S_1$, then  we have:
\begin{equation*}
\begin{split}
    (xz)\psi &= (xy^{-1})\psi =  (byy^{-1})\psi = (b) \psi (yy^{-1})\psi \\
    &= (b) \psi (y)\psi (y^{-1})\psi = (by) \psi (y^{-1})\psi \\
    &= (byy^{-1}y) \psi (y^{-1})\psi = (xy^{-1}y) \psi (y^{-1})\psi \\
    &= (x) \psi (y^{-1}yy^{-1}) \psi = (x) \psi (y^{-1}) \psi = (x) \psi (z) \psi.
\end{split}
\end{equation*}

\noindent On the other hand, if $ xy^{-1} \in S'_1 \smallsetminus S_1$, then, using the right rich ampleness of $S'_1$, we first write
\[
xz = xy^{-1} = (x^{-1})^{-1}y^{-1} = (x^{-1})^{-1}x^{-1} a.
\]
Note that $a\in S'_1 \smallsetminus S_1$, for otherwise we get $xy^{-1} \in S_1$, a contradiction. Now, one may calculate:
\begin{equation*}
\begin{split}
 (xz)\psi &=  (xy^{-1})\psi = ((x^{-1})^{-1}y^{-1}) \psi \\
 &= ((x^{-1})^{-1}x^{-1} a)\psi = ((x^{-1})^{-1}x^{-1})\psi (a)\psi \\
 &= (xx^{-1})\psi (a)\psi = (x)\psi (x^{-1})\psi (a)\psi \\
 &= (x)\psi (x^{-1}a)\psi = (x)\psi (x^{-1}xx^{-1}a)\psi  \\
 &= (x)\psi (x^{-1}xy^{-1})\psi = (x)\psi (x^{-1}x)\psi(y^{-1})\psi \\
 &= (xx^{-1}x)\psi(y^{-1})\psi = (x)\psi(y^{-1})\psi \\
 &= (x) \psi(z)\psi
\end{split}
\end{equation*}

\noindent This completes the proof that $\psi:S_{1} \cup S'_{1} \longrightarrow S_{2} \cup S'_{2}$ is an isomorphism of semigroups. That $(S;T_{1},T_{2})$ is weakly embeddable follows from Proposition \ref{VG}. The amalgam $(S;T_{1},T_{2})$, however, fails to embed strongly by Propositions \ref{Prop1} and \ref{Prop2}, and Theorem \ref{Non-embeddable1} (alternatively, Theorem \ref{Non-embeddable2}).
\end{proof}

\begin{corollary}
    Consider an amalgam $\mathcal{A} = (S;G_{1},G_{2};\phi_{1}, \phi_{2})$, in which $G_{1}$ and $G_{2}$ are groups and $S$ is not a group. Let $S\phi_{i}$ be rich ample in $G_{i}$ for each $i\in \{1,2\}$. Then $\mathcal{A}$ is weakly embeddable in a group.
\end{corollary}

\begin{proof}
   By the above theorem $(S;G_1,G_2)$ weakly embeds in an inverse semigroup, say $W$. Let $\psi_1:G_1 \longrightarrow W$ and $\psi_1:G_1 \longrightarrow W$ be the embedding monomorphisms. Then, by the properties of a pushout, there exists a unique homomorphism
   \[
   \psi:G_1*_SG_2 \longrightarrow W,
   \]
   such that $\eta_i \psi = \psi_i$ for $1 \leq i \leq 2$, where $\eta_i$ are as given in Figure 1. Clearly, $(S;G_1,G_2)$ embeds in $(G_1*_SG_2)\psi$, which is a subgroup of $W$ by the uniqueness of inverses in an inverse semigroup.
\end{proof}

\subsection{Connection with dominions}

Let $U$ be a subsemigroup of a semigroup $S$. Then, recall, for instance from \cite{Howie}, that an element $d \in S$ is said to be dominated by $U$ if for all homomorphisms $f,g: S \longrightarrow T$ with $f |_{U} = g|_{U}$ we have $(d)f=(d)g$. The set Dom$_{S}U$ of all elements of $S$ dominated by $U$ is a subsemigroup of $S$, called the  \textit{dominion} of $U$ in $S$.

\begin{proposition}    
    An ample subsemigroup $S$ of an inverse semigroup $T$ is rich ample in $T$ if and only if Dom$_{T}S=S \cup S'$.
\end{proposition}

\begin{proof}
($\implies$) Let $S$ be rich ample in $T$. Then by Proposition \ref{Invhull1}, the inverse hull of $S$ in $T$ is $S \cup S'$. Now, it follows from Proposition 1 of \cite{NSA} that Dom$_{T}S=S \cup S'$.\vspace{2pt}

\noindent ($\impliedby$) Assume that Dom$_{T}S=S \cup S'$, with $S$ being an ample in $T$. Then by Proposition 1 of \cite{NSA} the inverse hull of $S$ in $T$ is $S \cup S'$. Consequently, $S$ is rich ample in $T$ by Proposition \ref{Invhull1}.
\end{proof}

\begin{remark}
    It follows from the above proposition and zigzag theorem (see, for instance, \cite{Howie}, Theorem 8.3.3) that $S$ is rich ample in $T$ if and only if for all $t\in T$, the equality $t \otimes 1 = 1 \otimes t$ in the tensor product $T \otimes_{S}T$ implies that $t \in S \cup S'$. 
\end{remark}

\section{Ultra-rich ample semigroups}

In this section we introduce a special class of rich ample semigroups, namely the ultra-rich ample semigroups, that stems from our earlier considerations. The reader may refer to Section 6 for some natural examples of rich and ultra-rich ample semigroups.

\begin{definitions}
\label{Def3}\upshape Let a subsemigroup $S$ be rich right (respectively, left) ample in an inverse semigroup $T$ such that the elements $a$ (respectively, $b$) given in Proposition \ref{Prop1} (respectively, \ref{Prop2}) are uniquely determined. Then we say that{\tiny \ \ }$S$ is \textit{ultra-rich right} (respectively, \textit{left}) \textit{ample} in $T$. We call $S$ \textit{ultra-rich ample} in $T$ if it is both ultra-rich right and ultra-rich left ample in $T$. We also recall \cite[page 21]{C&P} that a semigroup is called \textit{unipotent}  if it contains precisely one idempotent. 
\end{definitions}

\begin{proposition}\label{USO1}
    If $S$ is ultra-rich right (left) ample subsemigroup of an inverse semigroup $T$ then $S$ is unipotent.
\end{proposition}

\begin{proof}
    Let $S$ be an ultra-rich right ample subsemigroup of an inverse semigroup $T$. Let $e_1, e_2 \in E(S)$. Then
    \[
    e_1^{-1}e_2 = e_1^{-1}e_1 a
    \]
for a unique $a \in S \cup S'$. Because
\[
e_1^{-1}e_2 = e_1^{-1}e_1 (e_1^{-1}e_2), \hspace{4pt} \text{and} \hspace{6pt} e_1^{-1}e_2 = e_1^{-1}e_1 (e_2),
\]
we get from the uniqueness of $a$:
\[
e_1e_2 = e_1^{-1}e_2 = e_2.
\]
Similarly, we may calculate $e_2e_1 = e_1$, whence by the commutativity of idempotents, $e_1 = e_2$, and we conclude that $E(S)$ is a singleton.\vspace{2pt}

Similarly, one can show that any ultra-rich left ample subsemigroup of an inverse semigroup is unipotent.
\end{proof}

\begin{corollary}\label{gpcrlry}
    Let $S$ be an inverse semigroup. Then $S$ is ultra-rich right (left) ample if and only if it is a group.
\end{corollary}

\begin{proof}
    Straightforward.
\end{proof}

\begin{proposition}\label{URA2}
    Let $S$ be a subsemigroup of an inverse semigroup $T$. Then $S$ is ultra-rich ample in $T$ if and only if it is unipotent and rich ample in $T$.
\end{proposition}

\begin{proof}
($\implies$) A proof of the direct part follows from Propositions \ref{USO1} and the fact that any ultra-rich ample subsemigroup of $T$ is rich ample in $T$.\vspace{2pt}

\noindent ($\impliedby$) Let $S$ be a rich ample subsemigroup of an inverse semigroup $T$ such that $E(S)=\{e\}$. This implies that $x^{-1}x=xx^{-1}=e$ for all $x \in S \cup S'$. Now, for any $y \in S$, there exists, by Condition (2) of Proposition \ref{Prop1}, an element $a \in S \cup S'$ such that
\begin{equation*}
    x^{-1}y=x^{-1}xa \;(=ea).
\end{equation*}
To prove the uniqueness of $a$ let
\begin{equation*}
    x^{-1}y=x^{-1}xc \;(=ec),
\end{equation*}
for some $c \in S \cup S'$. Then
\begin{equation*}    
a=(aa^{-1})a=ea=x^{-1}xa=x^{-1}y=x^{-1}xc=ec=(cc^{-1})c=c.
\end{equation*}
This proves that $U$ is ultra-rich right ample. Similarly, one can show that $U$ is ultra-rich left ample.
\end{proof}
 
\begin{lemma}
A subsemigroup $S$ of an inverse semigroup $T$ is ultra-rich right (respectively, left) ample if and only if $S'$ is an ultra-rich left (respectively, right) ample subsemigroup of $T$.
\end{lemma}

\begin{proof}
Straightforward.
\end{proof}

\begin{corollary}\label{gphull}
Let $S$ be an ultra-rich ample subsemigroup of an inverse semigroup $T$. Then the inverse hull $S\cup S'$ of $S$ is a subgroup of $T$.
\end{corollary}

\begin{proof}
It follows from Proposition \ref{Invhull1} that $S\cup S'$ is an inverse subsemigroup of $T$. Because $S$ and $S'$ are ultra-rich ample, such is $S\cup S'$. The proof now follows from Corollary \ref{gpcrlry}.
\end{proof}

\begin{corollary}
    Let $U$ be an ultra-rich ample subsemigroup of an inverse semigroup $S$ and $\phi:S \longrightarrow T$ be a homomorphism of (inverse) semigroups. Then $U\phi$ is an ultra-rich ample subsemigroup of $T$.
\end{corollary}

\begin{proof}
    Note that $U\phi$ is rich ample by Remark \ref{4.8}. On the other hand, being the image of a group, $(U \cup U')\phi$ is a subgroup of $T$. So, $(U \cup U')\phi$, and consequently $U\phi$,  must contain a unique idempotent. The corollary now follows by Proposition \ref{URA2}. 
\end{proof}

It is an easy exercise to show that every subsemigroup of a finite group $G$ is a subgroup of $G$. Also, recall from Corollary \ref{gphull} that the inverse hull of an ultra-rich ample subsemigroup of an inverse semigroup $T$ is a subgroup of $T$. Consequently, we make the following observation.

\begin{remark}\label{finite}
Assume that $S$ is an ultra-rich ample subsemigroup of an inverse semigroup $T$. If $T$ (or $S$) is finite then $S$ is a subgroup of $T$.
\end{remark}

\section{Examples}

By an (ultra-) rich (right, left) ample semigroup we shall mean a semigroup that is such in some inverse oversemigroup. Clearly every group is an inverse as well as an ultra-rich ample semigroup, whereas every inverse semigroup is a rich ample and hence an ample semigroup. In view of Remark \ref{finite}, we shall consider, in the first five examples, the infinite multiplicative inverse monoid $\mathbb{Q}$ of rational numbers. Examples 6--9 will concern the finite case.
\begin{enumerate}
    \item It is straightforward to observe that the (non-inverse) multiplicative monoid $\mathbb{N}$ of natural numbers is ample but not rich (right, left) ample in $\mathbb Q$.
    \item On the other hand, if $\mathbb{Q}^{+}_{1}$ denotes the submonoid of $\mathbb{Q}$ containing the unit $1$ and the elements of the form $p/q$ such that $p$ and $q$ are relatively prime positive integers with $p>q$. Then $\mathbb{Q}^+_{1}$ is a non-inverse ultra-rich ample submonoid of $\mathbb{Q}$.
    \item The non-inverse monoid $\mathbb{Q}^{+}_{1} \cup \{0\}$ is rich ample but not ultra-rich (right, left) ample in $\mathbb{Q}$.
    \item The inverse monoid $\mathbb{Q}$ is not ultra-rich ample in itself.
    \item The subgroup $\mathbb{Q}\smallsetminus \{0\}$ is inverse, a well as ultra-rich ample submonoid of $\mathbb{Q}$.
    \item Let $\mathcal{I}_{n}$, $n\geq 2$, denote the symmetric inverse monoid over a finite chain ${C}_{n}$: $1< \cdots<n$ of natural numbers. Let $\mathcal{DI}_{n}$ and $\mathcal{DI}^{+}_{n}$ denote the (non-inverse) ample submonoids of $\mathcal{I}_{n}$ comprising, respectively, the order decreasing and order increasing partial bijections \cite{NSA}. To see $\mathcal{DI}_{n}$ and $\mathcal{DI}^{+}_{n}$ are rich (right, left) ample in $\mathcal{I}_{n}$ for $n \leq 3$ one simply needs to effectuate some computations. We show that $\mathcal{DI}_{n}$ and $\mathcal{DI}^{+}_{n}$ are not rich (right, left) ample in $\mathcal{I}_{n}$ for all $n \geq 4$. To this end, consider the following elements of $\mathcal{DI}_{4}$:
\begin{equation*}
\begin{split}
u &=\{(4,3),(3,2)\}\text{,} \\ v&=\{(4,4),(3,1)\}.
\end{split}
\end{equation*}
Then
\begin{equation*}
u^{-1}v =\{(3,4),(2,1)\}\notin U \cup U',
\end{equation*}
implying that $\mathcal{DI}_{4}$ is not rich right ample in $\mathcal{I}_{n}$. Similarly, one can show that $\mathcal{DI}_{5}$ is not rich left ample in $\mathcal{I}_{n}$. That $\mathcal{DI}^{+}_{n}$ is not rich (right, left) ample in $\mathcal{I}_{n}$ follows from Lemma \ref{U&U'}.
Because $\mathcal{DI}_{m}$ (respectively, $\mathcal{DI}^{+}_{m}$) is contained in $\mathcal{DI}_{n}$ (respectively, $\mathcal{DI}^{+}_{n}$) for $m \leq n$, it follows that $\mathcal{DI}_{n}$ and $\mathcal{DI}^{+}_{n}$ are not rich (right, left) ample in $\mathcal{I}_{n}$ for all $n \geq 4$.

\item Let $\mathcal{I}_{n}$ ($n\geq 2$) denote the inverse monoid considered in the previous example. As every inverse subsemigroup of an inverse semigroup is rich ample, the members of the chain
\begin{equation*}
\mathbb{L}_{\text{INV}}:\{\iota\} \subseteq \mathcal{OI}_{n}\subset \mathcal{RI}_{n}\subset \mathcal{I}_{n}^{\prime }\subset \mathcal{I}_{n}
\end{equation*}
from Example \ref{Example3.3} are rich ample in $\mathcal{I}_{n}$. However, because $\mathcal{OI}_{n}$, $\mathcal{RI}_{n}$, $\mathcal{I}'_{n}$ and $\mathcal{I}_{n}$ are not unipotent, by Proposition \ref{USO1} none of them is ultra-rich (right, left) ample in $\mathcal{I}_{n}$.

\item The symmetric group $S_{n}$ is ultra-rich ample in $\mathcal{I}_{n}$.

\item If $S$ is a group then it is straightforward to observe that every monogenic subsemigroup of $S$ is ultra-rich ample. Thus, the monogenic subsemigroups of $S_{\mathbb{N}}$ are ultra-rich ample in $\mathcal{I}_{\mathbb{N}}$. 

\end{enumerate}

\noindent Based on the above examples, we have the following containment diagram for various classes of right ample subsemigroups of an inverse semigroup, considered in earlier sections. There exist similar diagrams for the corresponding classes of left ample and (two-sided) ample subsemigroups of an inverse semigroup. (The oval representing ultra-rich ample semigroups disappears in the finite case.)

\begin{figure}[htbp]
\centering
\includegraphics [width=100mm]{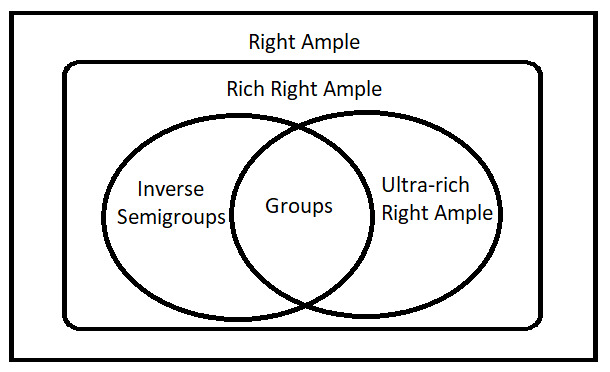}
\end{figure}\label{D2}

\begin{acknowledgement}
\upshape The author is thankful to Professor Victoria Gould for her comments that were received during his visit to the University of York, UK. I am also grateful to Professor Valdis Laan and Dr. Ülo Reimaa for their valuable comments on the earlier drafts of this article. This research has been supported by Estonian Science Foundation's grants PRG1204. 
\end{acknowledgement}

\end{document}